\documentclass[12pt,a4paper,reqno]{amsart} 

\usepackage[T1]{fontenc}       

\usepackage{amsfonts}          

\usepackage{amsmath}

\usepackage{amssymb}

\usepackage{overpic}
\usepackage{euscript} 
\usepackage{eurosym}
\usepackage{comment}
\usepackage{mathrsfs} 
\usepackage{ifthen}
\usepackage{esint} 
\usepackage{color}

\long\def\symbolfootnote[#1]#2{\begingroup%
\def\thefootnote{\fnsymbol{footnote}}\footnote[#1]{#2}\endgroup}

{\qed\vspace{5pt}}

\usepackage{amsthm}

\newtheoremstyle{lause}
{5pt}
{5pt}
{\slshape}
{\parindent}
{\bfseries}
{.}
{.5em}
{}

\theoremstyle{lause}

\newtheoremstyle{maaritelma}
{5pt}
{5pt}
{\rmfamily}
{\parindent}
{\bfseries}
{.}
{.5em}
{}

\theoremstyle{maaritelma}
\newtheoremstyle{lause}
{5pt}
{5pt}
{\slshape}
{\parindent}
{\bfseries}
{.}
{.5em}
{}

\theoremstyle{lause}
\newtheorem{theorem}{Theorem}[section]
\newtheorem{lemma}[theorem]{Lemma}

\newtheorem{corollary}[theorem]{Corollary}

\newtheoremstyle{maaritelma}
{5pt}
{5pt}
{\rmfamily}
{\parindent}
{\bfseries}
{.}
{.5em}
{}

\theoremstyle{maaritelma}
\newtheorem{definition}[theorem]{Definition}

\newtheorem{example}[theorem]{Example}
\newtheorem{remark}[theorem]{Remark}

\numberwithin{equation}{section}

\begin{document}

\thispagestyle{empty}

\begin{center}

{\large{\textbf{On Fuglede's problem on pseudo-balayage\\ for signed Radon measures of infinite energy}}}

\vspace{18pt}

\textbf{Natalia Zorii}

\vspace{18pt}

\emph{In memory of Bent Fuglede (1925-2023)}\vspace{8pt}

\footnotesize{\address{Institute of Mathematics, Academy of Sciences
of Ukraine, Tereshchenkivska~3, 02000, Kyiv, Ukraine\\
natalia.zorii@gmail.com }}

\end{center}

\vspace{12pt}

{\footnotesize{\textbf{Abstract.} For suitable kernels on a locally compact space, we develop a theory of inner (outer) pseudo-balayage of quite general signed Radon measures (not necessarily of finite energy) onto quite general sets (not necessarily closed). Such investigations were initiated in Fuglede's study (Anal.\ Math., 2016), which was, however, mainly concerned with the outer pse\-udo-bal\-ay\-age of positive measures of finite energy. The results thereby obtained solve Fuglede's problem, posed to the author in a private correspondence (2016), whether his theory could be extended to measures of infinite energy. An application of this theory to weighted minimum energy problems is also given.}}
\symbolfootnote[0]{\quad 2010 Mathematics Subject Classification: Primary 31C15.}
\symbolfootnote[0]{\quad Key words: Radon measures on a locally compact space; energy, consistency, and maximum principles; inner and outer pseudo-balayage.
}

\vspace{6pt}

\markboth{\emph{Natalia Zorii}} {\emph{On Fuglede's problem on pseudo-balayage for signed Radon measures of infinite energy}}

\section{Introduction and general conventions}\label{sec-intr} This paper deals with the theory of potentials on a locally compact (Hausdorff) space $X$ with respect to a kernel $\kappa$, {\it a kernel\/} being thought of as a symmetric, lower semi\-con\-tin\-uous (l.s.c.) function $\kappa:X\times X\to[0,\infty]$. In more details, for suitable $X$ and $\kappa$, we establish a theory of {\it inner} pse\-udo-bal\-ay\-age $\widehat{\omega}^A$ of quite a general {\it signed\/} Radon measure $\omega$ (not necessarily of finite energy) to quite a general set $A$ (not necessarily closed), where $\widehat{\omega}^A$ is introduced as the solution to the problem of minimizing the Gauss functional
\[\int\kappa(x,y)\,d(\mu\otimes\mu)(x,y)-2\int\kappa(x,y)\,d(\mu\otimes\omega)(x,y),\]
$\mu$ ranging over all positive measures of finite energy concentrated on the set $A$.

As a by-product, we provide a generalization of Fuglede's theory \cite{Fu5} of {\it outer} pse\-udo-bal\-ay\-age $\widehat{\omega}^{\,*A}$, which was mainly concerned with {\it positive} $\omega$ of {\it finite} energy. The results thereby obtained solve Fuglede's problem, posed to the author in a private correspondence, whether his theory could be extended to measures of infinite energy.

The concept of inner pse\-udo-bal\-ay\-age is further shown to be a powerful tool in the well-known inner Gauss variational problem, see Section~\ref{sec-appl} for the results obtained.

\subsection{General conventions}\label{sec-intr1} In what follows, a locally compact space $X$ is assumed to be second-countable. Then it is
{\it $\sigma$-compact} (that is, representable as a countable union of compact sets \cite[Section~I.9, Definition~5]{B1}), see \cite[Section~IX.2, Corollary to Proposition~16]{B3}; and hence
the concept of negligibility coincides with that of local negligibility, see \cite[Section~IV.5, Corollary~3 to Proposition~5]{B2}.

The reader is expected to be familiar with principal concepts of the theory of measures and integration on a locally compact space. For its exposition we refer to Bourbaki \cite{B2} or Edwards \cite{E2}; see also Fuglede \cite{F1} for a brief survey.

We denote by $\mathfrak M$ the linear space of all (real-valued Radon) measures $\mu$ on $X$, equipped with the {\it vague} (={\it weak\/$^*$}) topology of pointwise convergence on the class $C_0(X)$ of all continuous functions $\varphi:X\to\mathbb R$ of compact support, and by $\mathfrak M^+$ the cone of all positive $\mu\in\mathfrak M$, where $\mu$ is {\it positive} if and only if $\mu(\varphi)\geqslant0$ for all positive $\varphi\in C_0(X)$. The space $X$ being second-countable, each $\mu\in\mathfrak M$ has a countable base of vague neighborhoods \cite[Lemma~4.4]{Z-arx}, and hence any vaguely bounded subset of $\mathfrak M$ has a {\it sequence} that is vaguely convergent in $\mathfrak M$ (cf.\ \cite[Section~III.1, Proposition~15]{B2}).

Given $\mu,\nu\in\mathfrak M$, the {\it mutual energy} and the {\it potential} are introduced by
\begin{align*}
  I(\mu,\nu)&:=\int\kappa(x,y)\,d(\mu\otimes\nu)(x,y),\\
  U^\mu(x)&:=\int\kappa(x,y)\,d\mu(y),\quad x\in X,
\end{align*}
respectively, provided the value on the right is well defined as a finite number or $\pm\infty$. For $\mu=\nu$, the mutual energy $I(\mu,\nu)$ defines the {\it energy} $I(\mu,\mu)=:I(\mu)$ of $\mu\in\mathfrak M$.

Throughout this paper, a kernel $\kappa$ is assumed to satisfy the {\it energy principle}, or equivalently to be {\it strictly positive definite}, which means that $I(\mu)\geqslant0$ for all (signed) $\mu\in\mathfrak M$, and moreover that $I(\mu)=0$ only for $\mu=0$. Then all (signed) measures of finite energy form a pre-Hil\-bert space $\mathcal E$ with the inner product $\langle\mu,\nu\rangle:=I(\mu,\nu)$ and the energy norm $\|\mu\|:=\sqrt{I(\mu)}$, cf.\ \cite[Lemma~3.1.2]{F1}. The topology on $\mathcal E$ introduced by means of this norm is said to be {\it strong}.

Another permanent requirement on $\kappa$ is that it satisfies the {\it consistency} principle, which means that the cone
$\mathcal E^+:=\mathcal E\cap\mathfrak M^+$ is {\it complete} in the induced strong topology, and that the strong topology on $\mathcal E^+$ is {\it finer} than the vague topology on $\mathcal E^+$; such a kernel is said to be {\it perfect} (Fuglede \cite{F1}). Thus any strong Cauchy sequence (net) $(\mu_j)\subset\mathcal E^+$ converges {\it both strongly and vaguely} to the same unique measure $\mu_0\in\mathcal E^+$, the strong topology on $\mathcal E$ as well as the vague topology on $\mathfrak M$ being Hausdorff.

We shall sometimes also need the {\it domination} and {\it Ugaheri maximum principles}, where the former means that for any $\mu\in\mathcal E^+$ and $\nu\in\mathfrak M^+$ with $U^\mu\leqslant U^\nu$ $\mu$-a.e., the same inequality holds on all of $X$; whereas the latter means that there is $h\in[1,\infty)$, depending on $X$ and $\kappa$ only, such that for each $\mu\in\mathcal E^+$ with $U^\mu\leqslant c_\mu$ $\mu$-a.e., where $c_\mu\in(0,\infty)$, we have $U^\mu\leqslant hc_\mu$ on all of $X$. When the constant $h$ is specified, we speak of {\it $h$-Ugaheri's maximum principle}, and when $h=1$, $h$-Ugaheri's maximum principle is referred to as {\it Frostman's maximum principle}. See \cite[Section~1.2]{O}.

See Example~\ref{rem:clas} for kernels satisfying some/all of the above principles.

\begin{example}\label{rem:clas} (i) The $\alpha$-Riesz kernel $\kappa_\alpha(x,y):=|x-y|^{\alpha-n}$ of order $\alpha\in(0,2]$, $\alpha<n$, on $\mathbb R^n$, $n\geqslant2$ (thus in particular the Newtonian kernel $\kappa_2(x,y)$ on $\mathbb R^n$, $n\geqslant3$), is perfect, and it satisfies the domination and Frostman maximum principles. See \cite[Theorems~1.10, 1.15, 1.18, 1.27, 1.29]{L}.

(ii) The same holds true for the associated $\alpha$-Green kernel on an arbitrary open subset of $\mathbb R^n$, $n\geqslant2$. See \cite[Theorems~4.6, 4.9, 4.11]{FZ}.

(iii) The ($2$-)Green kernel on a planar Greenian set is likewise strictly positive definite \cite[Section~I.XIII.7]{Doob} and perfect \cite{E}, and it fulfills the domination and Frostman maximum principles (see \cite[Theorem~5.1.11]{AG} or \cite[Section~I.V.10]{Doob}).

(iv) The restriction of the logarithmic kernel $-\log\,|x-y|$ to a closed disc in $\mathbb R^2$ of radius ${}<1$ satisfies the energy and Frostman maximum principles \cite[Theorems~1.6, 1.16]{L}, and hence it is perfect \cite[Theorem~3.4.2]{F1}.\footnote{However, the domination principle then fails in general; it does hold only in a weaker sense where the measures $\mu,\nu$ involved in the ab\-ove-quo\-ted definition meet the additional requirement $\nu(\mathbb R^2)\leqslant\mu(\mathbb R^2)$, see \cite[Theorem~II.3.2]{ST}.}

(v) The $\alpha$-Riesz kernels of order $2<\alpha<n$ on $\mathbb R^n$, $n\geqslant2$, are likewise perfect, and satisfy $h$-Ugaheri's maximum principle with $h:=2^{n-\alpha}$, see \cite[Theorems~1.5, 1.15, 1.18]{L}.

(vi) The Deny kernels, defined with the aid of Fourier transformation (see $(A)$ in \cite[Section~1]{De2}, cf.\ \cite[Section~VI.1.2]{L}), are perfect as well.\end{example}

For the {\it inner} and {\it outer} capacities of a set $A\subset X$, denoted by $c_*(A)$ and $c^*(A)$, respectively, we refer to Fuglede \cite[Section~2.3]{F1}. If $A$ is capacitable (e.g.\ open or compact), we write $c(A):=c_*(A)=c^*(A)$. A proposition $\mathcal P(x)$ involving a variable point $x\in X$ is said to hold {\it quasi-everywhere} ({\it q.e.}) on $A$ if the set $N$ of all $x\in A$ where $\mathcal P(x)$ fails, is of outer capacity zero. Replacing here $c^*(N)=0$ by $c_*(N)=0$, we arrive at the concept of {\it nearly everywhere} ({\it n.e.}) on $A$. See \cite[p.~153]{F1}.

For any $A\subset X$, we denote by $\mathfrak C_A$ the upward directed set of all compact subsets $K$ of $A$, where $K_1\leqslant K_2$ if and only if $K_1\subset K_2$. If a net $(x_K)_{K\in\mathfrak C_A}\subset Y$ converges to $x_0\in Y$, $Y$ being a topological space, then we shall indicate this fact by writing
\begin{equation*}x_K\to x_0\text{ \ in $Y$ as $K\uparrow A$}.\end{equation*}

Given $A\subset X$, let $\mathfrak M^+(A)$ denote the set of all $\mu\in\mathfrak M^+$ {\it concentrated on} $A$, which means that $A^c:=X\setminus A$ is $\mu$-negligible, or equivalently that $A$ is $\mu$-mea\-s\-ur\-ab\-le and $\mu=\mu|_A$, $\mu|_A$ being the trace of $\mu$ to $A$, cf.\ \cite[Section~V.5.7]{B2}. (If a set $A$ is closed, then $\mu\in\mathfrak M^+(A)$ if and only if $S(\mu)\subset A$, where $S(\mu)$ denotes the support of $\mu$.)

Also define $\mathcal E^+(A):=\mathcal E\cap\mathfrak M^+(A)$. As seen from \cite[Lemma~2.3.1]{F1},
\begin{equation}\label{iff}
 c_*(A)=0\iff\mathcal E^+(A)=\{0\}\iff\mathcal E^+(K)=\{0\}\quad\text{for every $K\in\mathfrak C_A$}.
\end{equation}

In what follows, let a set $A\subset X$ be fixed. To avoid trivialities, suppose that
\begin{equation}\label{non0}
 c_*(A)>0.
\end{equation}
While approximating $A$ by $K\in\mathfrak C_A$, we may therefore only consider $K$ with $c(K)>0$.

Yet another permanent condition imposed on a set $A$ is that the class $\mathcal E^+(A)$ is {\it closed\/} in the induced strong topology.\footnote{As shown in \cite[Theorem~2.13]{Z-Oh}, this occurs, for instance, if $A$ is {\it quasiclosed} ({\it quasicompact\/}), that is, if $A$ can be approximated in outer capacity by closed (compact) sets (Fuglede \cite[Definition~2.1]{F71}). We note that, in general, a quasiclosed set is not Borel measurable (Fuglede, private correspondence).} Being, therefore, a strongly closed subcone of the strongly complete cone $\mathcal E^+$, the cone $\mathcal E^+(A)$ is likewise {\it strongly complete}.

Also fix $\omega\in\mathfrak M$, $\omega\ne0$. Unless $\omega\in\mathcal E$, assume that (a)--(c) are fulfilled, where:
\begin{itemize}
  \item[(a)] {\it $\omega^+$ is bounded, i.e.\ $\omega^+(X)<\infty$.}\footnote{$\omega^+$ and $\omega^-$ denote respectively the positive and negative parts of $\omega$ in the Hahn--Jor\-dan decomposition, see \cite[Section~III.1, Theorem~2]{B2}. We also write $|\omega|:=\omega^++\omega^-$.}
 \item[(b)] {\it For every $K\in\mathfrak C_A$, $U^{\omega^+}\bigl|_K$ is upper semicontinuous} ({\it u.s.c.}), {\it hence continuous;  and moreover}
     \begin{equation}\label{MA}
      M_A:=\sup_{x\in A}\,U^{|\omega|}(x)<\infty.
     \end{equation}
  \item[(c)] {\it $\kappa$ satisfies $h$-Ugaheri's maximum principle.}
\end{itemize}

The above-mentioned assumptions on a locally compact space $X$, a kernel $\kappa$, a set $A$, and a measure $\omega$ will usually not be repeated henceforth.

Note that, if $\omega\in\mathcal E$, then the only requirement imposed on $\kappa$ is its perfectness. Being actually investigated in the author's recent paper \cite{Z-Expo}, the case where $\omega\in\mathcal E$ is incorporated in the present study only for the sake of completeness.

\section{Inner and outer pseudo-balayage}\label{sec-intr2} It is often convenient to treat the above measure $\omega\in\mathfrak M$ as a charge creating the {\it external field}
\begin{equation*}
f:=-U^\omega.
\end{equation*}
If $\omega\in\mathcal E$, then $f$ is well defined and finite q.e.\ (hence, n.e.) on $X$, see \cite[Corollary to Lemma~3.2.3]{F1}, while otherwise $|f|$ is well defined and bounded on $A$, cf.\ (\ref{MA}).

We denote by $\mathcal E^+_f(A)$ the class of all $\mu\in\mathcal E^+(A)$ such that the external field $f$ is $\mu$-integrable, i.e.\ $f\in\mathcal L^1(\mu)$ (see \cite[Chapter~IV, Sections~3, 4]{B2}), and define
\begin{equation}\label{W}
\widehat{w}_f(A):=\inf_{\mu\in\mathcal E^+_f(A)}\,I_f(\mu),
\end{equation}
where $I_f(\mu)$ is the so-called {\it Gauss functional},\footnote{For the terminology used here, see e.g.\ \cite{L,O}. In constructive function theory, $I_f(\mu)$ is often
referred to as {\it the $f$-weighted energy}, see e.g.\ \cite{BHS,ST}.} introduced by means of the formula
\begin{equation}\label{If}
I_f(\mu):=\|\mu\|^2+2\int f\,d\mu=\|\mu\|^2-2\int U^\omega\,d\mu\in(-\infty,\infty).
\end{equation}
Then
\begin{equation}\label{West}-\infty\leqslant\widehat{w}_f(A)\leqslant0,\end{equation}
the upper estimate being caused by the fact that $0\in\mathcal E^+_f(A)$, while $I_f(0)=0$.

Observe that in the case $\omega\in\mathcal E$, we actually have
\begin{equation}\label{Ef}
\mathcal E^+_f(A)=\mathcal E^+(A),
\end{equation}
which is clear from the Cauchy--Schwarz (Bunyakovski) inequality, applied to $\omega\in\mathcal E$ and $\mu\in\mathcal E^+(A)$, while otherwise $\mathcal E^+_f(A)$ includes all {\it bounded} $\mu\in\mathcal E^+(A)$, the latter being obvious from (\ref{MA}) by making use of \cite[Section~IV.3, Corollary~2 to Theorem~4]{B2}. The same Corollary from \cite{B2} also implies that $\mathcal E^+_f(A)$ is a {\it convex cone}.

\begin{theorem}\label{th1}
There is one and the same measure $\widehat{\omega}^A\in\mathcal E^+_f(A)$, called the inner pseudo-balayage of $\omega$ onto $A$, that satisfies any one of the following three assertions.
\begin{itemize}
  \item[{\rm(i$_1$)}] $\widehat{\omega}^A$ is the unique solution to problem {\rm(\ref{W})}, that is, $\widehat{\omega}^A\in\mathcal E^+_f(A)$ and
      \begin{equation}\label{e-i}
       I_f(\widehat{\omega}^A)=\min_{\mu\in\mathcal E^+_f(A)}\,I_f(\mu)=\widehat{w}_f(A)\in(-\infty,0].
      \end{equation}
  \item[{\rm(ii$_1$)}] $\widehat{\omega}^A$ is the only measure in $\mathcal E^+_f(A)$ having the two properties
 \begin{align}\label{def1'}
&\int U^{\widehat{\omega}^A-\omega}\,d\mu\geqslant0\quad\text{for all $\mu\in\mathcal E^+_f(A)$},\\
&\int U^{\widehat{\omega}^A-\omega}\,d\widehat{\omega}^A=0.\label{def2'}
\end{align}
\item[{\rm(iii$_1$)}] $\widehat{\omega}^A$ is the only measure in $\mathcal E^+_f(A)$ having the two properties\footnote{In view of (\ref{def1}), $U^{\widehat{\omega}^A}\geqslant U^\omega$ $\widehat{\omega}^A$-a.e.\ (Lemma~\ref{l1}); hence, (\ref{def2}) can be replaced by the
    apparently weaker relation $U^{\widehat{\omega}^A}\leqslant U^\omega$ $\widehat{\omega}^A$-a.e. Similarly, (\ref{def2'}) can be replaced by $\int U^{\widehat{\omega}^A-\omega}\,d\widehat{\omega}^A\leqslant0$.}
  \begin{align}\label{def1}
  U^{\widehat{\omega}^A}&\geqslant U^\omega\quad\text{n.e.\ on $A$},\\
  U^{\widehat{\omega}^A}&=U^\omega\quad\text{$\widehat{\omega}^A$-a.e.}\label{def2}
  \end{align}
\end{itemize}
\end{theorem}

\begin{remark}\label{rem1} Assume for a moment that the kernel $\kappa$ satisfies the domination principle. If moreover $\omega$ is {\it positive}, then, by virtue of (\ref{def1}) and (\ref{def2}),
\begin{align}
U^{\widehat{\omega}^A}&=U^\omega\quad\text{n.e.\ on $A$},\label{ineq1}\\
U^{\widehat{\omega}^A}&\leqslant U^\omega\quad\text{on $X$},\notag
\end{align}
so that $\widehat{\omega}^A$ serves as the {\it inner balayage} $\omega^A$ of $\omega$ onto $A$ (cf.\ \cite{Z-arx1,Z-Fin}). In other words, for positive $\omega$, the concept of inner pseudo-balayage, introduced by means of Theorem~\ref{th1}, represents a natural extension of the concept of inner balayage to kernels that do not satisfy the domination principle. However, {\it this is no longer so if $\omega$ is signed}, for the inner balayage of a signed measure is usually defined by linearity:
\[\omega^A=(\omega^+)^A-(\omega^-)^A,\]
and hence it must likewise be signed. Thus, for signed measures $\omega$ and kernels $\kappa$ satisfying the domination principle, the theory of inner pseudo-balayage provides an alternative approach to inner balayage that is {\it not} equivalent to the classical one. For further illustrations of this non-equ\-iv\-al\-ence, see Remarks~\ref{rem2} and \ref{rem3} below.
\end{remark}

\begin{remark}\label{rem2} It is obvious from Theorem~\ref{th1} that for any $q\in(0,\infty)$,
\[\widehat{q\omega}^A=q\widehat{\omega}^A.\]
However, this fails to hold whenever $q\in(-\infty,0)$. In fact, if $\omega=-\omega^-\ne0$,
then
\[\widehat{\omega}^A=\widehat{(-\omega^-)}^A=0,\]
because $\widehat{w}_f(A)=0=I_f(0)$, the former equality being clear from $0\in\mathcal E^+_f(A)$ and
\[I_f(\mu)=\|\mu\|^2+2\int U^{\omega^-}\,d\mu\geqslant0\quad\text{for all $\mu\in\mathcal E^+_f(A)$}.\]
\end{remark}

\begin{remark}\label{rem3} In general, equality does not prevail in (\ref{def1}) (as it does for the inner balayage, cf.\ (\ref{ineq1})), which is seen by taking $\omega=-\omega^-\ne0$, cf.\ Remark~\ref{rem2}.
\end{remark}

\begin{theorem}\label{th1'}If $A$ is Borel, then Theorem~{\rm\ref{th1}} remains valid with "n.e.\ on $A$" replaced by "q.e.\ on $A$". The measure $\widehat{\omega}^{\,*A}$, thereby uniquely determined, is said to be the outer pseudo-balayage of $\omega$ onto $A$. Actually,
\begin{equation}\label{io}
\widehat{\omega}^{\,*A}=\widehat{\omega}^A.
\end{equation}
\end{theorem}

\begin{remark}It follows that for Borel $A$, Remarks~\ref{rem1}--\ref{rem3} remain valid with $\widehat{\omega}^A$, $\omega^A$, and "n.e.\ on $A$" replaced by $\widehat{\omega}^{\,*A}$, $\omega^{\,*A}$, and "q.e.\ on $A$", respectively, where $\omega^{\,*A}$ denotes the {\it outer} balayage of $\omega$ onto $A$ (cf.\ \cite[Theorem~9.4]{Z-arx1}, \cite[Theorem~1.5]{Z-Fin}).\end{remark}

\begin{remark}If $A$ is quasiclosed and $\omega\in\mathcal E^+$, then the existence of the outer pseu\-do-bal\-ayage $\widehat{\omega}^{\,*A}$ was established by Fuglede \cite[Theorem~4.10]{Fu5}. The methods exploited in the present work are substantially different from those in \cite{Fu5}, which enables us to generalize Fuglede's result to {\it signed} $\omega$ of {\it infinite} energy (Theorem~\ref{th1'}) as well as to develop the theory of inner pseudo-balayage $\widehat{\omega}^A$ (Theorem~\ref{th1}). In addition, we also study the strong and the vague continuity of $\widehat{\omega}^A$ and $\widehat{\omega}^{\,*A}$ under approximation of $A$ by monotone families of sets (Theorems~\ref{conv=ps} and \ref{conv=ps'}).\end{remark}

\section{Proofs of Theorems~\ref{th1} and \ref{th1'}}

We quote for future reference some known facts, useful in the sequel. For any $\nu\in\mathfrak M^+$, let $\nu^*(E)$, resp.\ $\nu_*(E)$, stand for the {\it outer}, resp.\ {\it inner}, measure of $E\subset X$. A subset of $X$ is said to be {\it universally measurable} if it is $\nu$-meas\-ur\-able for every $\nu\in\mathfrak M^+$, see \cite[Section~V.3.4]{B2}.

\begin{lemma}\label{l1}
For any $E\subset X$, any $\mu\in\mathcal E^+(E)$, and any universally measurable $U\subset X$ such that $c_*(E\cap U)=0$, we have $\mu^*(E\cap U)=0$.
\end{lemma}

\begin{proof}
Being the intersection of universally measurable $U$ and $\mu$-measurable $E$, the set $E\cap U$ is $\mu$-measurable. Since the space $X$ is $\sigma$-compact, it is therefore enough to show that $\mu_*(E\cap U)=0$, which is however obvious from (\ref{iff}).
\end{proof}

\begin{lemma}\label{str}
For any $E\subset X$ and any universally measurable $U_j$, $j\in\mathbb N$,
\[c_*\Bigl(\bigcup_{j\in\mathbb N}\,E\cap U_j\Bigr)\leqslant\sum_{j\in\mathbb N}\,c_*(E\cap U_j).\]\end{lemma}

\begin{proof}Since a strictly positive definite kernel is pseudo-positive, cf.\ \cite[p.~150]{F1}, the lemma follows directly from Fuglede \cite{F1} (see Lemma~2.3.5 and the remark after it). For the Newtonian kernel on $\mathbb R^n$, this goes back to Cartan \cite[p.~253]{Ca2}.\end{proof}

\begin{lemma}\label{l2}If a sequence $(\nu_j)\subset\mathcal E$ converges strongly to $\nu_0$, then there exists a subsequence $(\nu_{j_k})$ whose potentials converge to $U^{\nu_0}$ pointwise n.e.\ on $X$.\end{lemma}

\begin{proof}See Fuglede \cite{F1}, the remark attached to Lemma~3.2.4.\end{proof}

\subsection{Proof of Theorem~\ref{th1}} We first note that the uniqueness of the solution $\widehat{\omega}^A$ to problem (\ref{W}) within $\mathcal E^+_f(A)$ can easily be established by means of standard arguments, based on the convexity of the cone $\mathcal E^+_f(A)$, the parallelogram identity in the pre-Hil\-bert space $\mathcal E$, and the energy principle (cf.\ \cite[Proof of Lemma~6]{Z5a}).

Fixing $\nu\in\mathcal E^+_f(A)$, we further show that (\ref{def1}) and (\ref{def2}) hold true for $\nu$ in place of $\widehat{\omega}^A$ if and only if so do (\ref{def1'}) and (\ref{def2'})~--- again for $\nu$ in place of $\widehat{\omega}^A$. It is enough to verify the "if" part of this claim, since the opposite is obvious by noting that for every $\mu\in\mathcal E^+_f(A)$, $U^{\widehat{\omega}^A-\omega}$ is then $\mu$-equivalent to a positive $\mu$-integrable function that is defined on all of $X$ (see Lemma~\ref{l1} and \cite[Section~IV.3, Corollary~2 to Theorem~4]{B2}; for the concept of {\it $\mu$-equivalent} functions, see \cite[Section~IV.2.4]{B2}).

Assuming, therefore, that (\ref{def1'}) and (\ref{def2'}) are fulfilled, suppose to the
contrary that (\ref{def1}) fails. Then there exists compact $K\subset A$ such that $U^\nu<U^\omega$ on
$K$ while $c(K)>0$, hence
$\int U^{\nu-\omega}\,d\tau<0$ for any $\tau\in\mathcal E^+(K)$, $\tau\ne0$,\footnote{Here we have used \cite[Chapter~IV]{B2} (see Proposition~10 in Section~1 and Theorem~1 in Section~2). Note that such a measure $\tau$ does indeed exist because of our convention (\ref{non0}), cf.\ (\ref{iff}), and also that $f\in\mathcal L^1(\tau)$, $\tau$ being bounded and of finite energy.\label{f1'}} which however contradicts (\ref{def1'}) for $\mu:=\tau$.
Thus (\ref{def1}) does indeed hold, and so, by Lemma~\ref{l1},
\begin{equation}\label{pr-in}
U^\nu\geqslant U^\omega\quad\text{$\nu$-a.e.}
\end{equation}

Now, assuming to the contrary that (\ref{def2}) (with $\nu$ in place of $\widehat{\omega}^A$) fails, we infer from (\ref{pr-in}) that there is
compact $Q\subset A$ such that $\nu(Q)>0$ while $U^\nu>U^\omega$ on $Q$. This together with (\ref{pr-in}) implies, again by use of the assertions from \cite{B2} mentioned in footnote~\ref{f1'}, that $\int
U^{\nu-\omega}\,d\nu>0$, which is however impossible because of (\ref{def2'}).

The equivalence thereby verified enables us to prove the statement on the uniqueness in each of (ii$_1$) and (iii$_1$). Indeed, assuming that both (\ref{def1'}) and (\ref{def2'}) hold for each of $\nu,\nu'\in\mathcal E^+_f(A)$, we get, by use of
\cite[Section~IV.4, Corollary~2 to Theorem~1]{B2},
 \[\langle\nu,\nu'\rangle\geqslant\int U^\omega\,d\nu'=\|\nu'\|^2,\quad\langle\nu',\nu\rangle\geqslant\int
 U^\omega\,d\nu=\|\nu\|^2,\]
and therefore
  \[0\leqslant\|\nu-\nu'\|^2=\bigl(\|\nu\|^2-\langle\nu',\nu\rangle\bigr)+
  \bigl(\|\nu'\|^2-\langle\nu,\nu'\rangle\bigr)\leqslant0,\]
whence $\nu=\nu'$, by the energy principle.

Consider first the case where $\omega\in\mathcal E$; then the Gauss functional has the form
\begin{equation*}
I_f(\mu)=\|\omega-\mu\|^2-\|\omega\|^2\quad\text{for all $\mu\in\mathcal E^+(A)$},
\end{equation*}
whence (\ref{Ef}). Thus the question on the existence of the solution $\widehat{\omega}^A$ to problem (\ref{W}) is reduced to
that on the existence of the orthogonal projection of $\omega$ onto $\mathcal E^+(A)$, i.e.
\begin{equation*}
 \widehat{\omega}^A\in\mathcal E^+(A)\quad\text{and}\quad\|\omega-\widehat{\omega}^A\|=\min_{\mu\in\mathcal E^+(A)}\,\|\omega-\mu\|.
\end{equation*}
The class $\mathcal E^+(A)$ being convex and strongly complete (Section~\ref{sec-intr1}), applying \cite{E2} (Theorem~1.12.3 and Proposition~1.12.4(2)) shows
that this orthogonal projection does exist, and it is uniquely characterized within $\mathcal E^+(A)$ by both
(\ref{def1'}) and (\ref{def2'}). (Here we have utilized the fact that $0\in\mathcal E^+(A)$ and that $\mathcal E^+(A)+\mathcal E^+(A)\subset\mathcal E^+(A)$.) In view of the equivalence of (ii$_1$) and (iii$_1$) proved above, this implies the theorem.\footnote{Cf.\ \cite[Theorem~2.1]{Z-Expo}.}

It thus remains to complete the proof in the case of (a)--(c).
We begin by showing that the solution $\widehat{\omega}^A$ to problem (\ref{W}) exists if and only if there is the (unique) measure
$\mu_0\in\mathcal E^+_f(A)$ satisfying (ii$_1$) (equivalently, (iii$_1$); see above), and then necessarily
\begin{equation}\label{st1}
\mu_0=\widehat{\omega}^A.
\end{equation}

Assume first that this $\widehat{\omega}^A$ exists. To verify (\ref{def1}), suppose to the contrary that there is compact $K\subset A$ with $c(K)>0$, such that $U^{\widehat{\omega}^A}<U^\omega$ on $K$. A straightforward verification then shows that for any $\tau\in\mathcal E^+(K)$, $\tau\ne0$, and any $t\in(0,\infty)$,\footnote{In (\ref{step1}) (as well as in
(\ref{eqpr4'}), (\ref{eqpr4''}), and (\ref{est})),
we use \cite[Section~IV.4, Corollary~2 to Theorem~1]{B2}. To this end, we observe that $f\in\mathcal L^1(\tau)$, for $\tau$ is bounded, being of compact support.}
 \begin{equation}\label{step1}
   I_f(\widehat{\omega}^A+t\tau)-I_f(\widehat{\omega}^A)=2t\int\bigl(U^{\widehat{\omega}^A}-U^\omega\bigr)\,d\tau+
   t^2\|\tau\|^2,
 \end{equation}
 and moreover, similarly as in the third paragraph of this proof,
 \[\int\bigl(U^{\widehat{\omega}^A}-U^\omega\bigr)\,d\tau<0.\]
 Thus, the value on the right in (\ref{step1}) (hence, also that on the left) must be ${}<0$ when $t>0$ is
 small enough, which is however impossible because of $\widehat{\omega}^A+t\tau\in\mathcal E^+_f(A)$.

Having thus established (\ref{def1}), in view of Lemma~\ref{l1} we obtain
 \begin{equation}\label{In}
 U^{\widehat{\omega}^A}\geqslant U^\omega\quad\text{$\widehat{\omega}^A$-a.e.}
 \end{equation}

Suppose now that (\ref{def2}) fails; then, in consequence of (\ref{In}), there is a compact set $Q\subset A$ with $\widehat{\omega}^A(Q)>0$, such that
 $U^{\widehat{\omega}^A}>U^\omega$ on $Q$. Denoting $\upsilon:=\widehat{\omega}^A|_Q$, we have
 $\widehat{\omega}^A-t\upsilon\in\mathcal E^+_f(A)$ for all $t\in(0,1)$, hence
 \begin{equation}\label{eqpr4'}
   I_f(\widehat{\omega}^A-t\upsilon)-I_f(\widehat{\omega}^A)=-2t\int\bigl(U^{\widehat{\omega}^A}-
   U^\omega\bigr)\,d\upsilon+t^2\|\upsilon\|^2,
 \end{equation}
 which is however again impossible when $t$ is getting sufficiently small.

For the "if" part of the above claim, assume that (ii$_1$) holds true for some (unique) $\mu_0\in\mathcal E_f^+(A)$. To show that then
 necessarily $\mu_0=\widehat{\omega}^A$, we only need to verify that
 \begin{equation}\label{N}
 I_f(\mu)-I_f(\mu_0)\geqslant0\quad\text{for any $\mu\in\mathcal E_f^+(A)$}.
 \end{equation}
 But obviously
  \begin{align}I_f(\mu)-I_f(\mu_0)&=\|\mu-\mu_0+\mu_0\|^2-2\int U^\omega\,d\mu-\|\mu_0\|^2+2\int U^\omega\,d\mu_0\notag\\
  {}&=\|\mu-\mu_0\|^2+2\int U^{\mu_0-\omega}\,d(\mu-\mu_0),\label{eqpr4''}\end{align}
 and applying (\ref{def1'}) and (\ref{def2'}), both with $\mu_0$ in place of $\widehat{\omega}^A$, gives (\ref{N}), whence (\ref{st1}).

To complete the proof of the theorem, it thus remains to establish the existence of the solution $\widehat{\omega}^A$ to problem (\ref{W}). To this end, assume first that $A:=K$ is {\it compact}. As $c(K)>0$, we may restrict ourselves to {\it nonzero} $\mu\in\mathcal E^+(K)$, cf.\ (\ref{iff}). For each of those $\mu$, there are $t\in(0,\infty)$ and $\tau\in\mathcal E^+(K)$ with $\tau(K)=1$ such that $\mu=t\tau$. The potential $U^{|\omega|}$ being bounded on $K$ according to (b),
  \begin{align}\notag
  I_f(\mu)&=t^2\|\tau\|^2-2t\int U^\omega\,d\tau\geqslant t^2\|\tau\|^2-2t\int U^{\omega^+}\,d\tau\\
  {}&\geqslant t^2c(K)^{-1}-2tM_K=t^2\bigl(c(K)^{-1}-2M_Kt^{-1}\bigr),\label{est}
  \end{align}
 $M_K\in(0,\infty)$ being introduced by (\ref{MA}) with $A:=K$. Thus, by virtue of (\ref{est}), $I_f(\mu)>0$ for all
 $\mu\in\mathcal E^+(K)$ having the property
  \[\mu(K)>2M_Kc(K)=:L_K\in(0,\infty).\]

  On the other hand, $\widehat{w}_f(K)\leqslant0$, cf.\ (\ref{West}). In view of the above, $\widehat{w}_f(K)$ would therefore be the same if $\mathcal E_f^+(K)$ in (\ref{W}) were replaced by
  \begin{equation}\label{LL}
 \mathcal E^+_{L_K}(K):=\bigl\{\mu\in\mathcal E^+(K):\ \mu(K)\leqslant L_K\bigr\}\quad\bigl({}\subset\mathcal E_f^+(K)\bigr).\end{equation}
 That is,
 \begin{equation}\label{L}
 \widehat{w}_f(K)=\inf_{\mu\in\mathcal E^+_{L_K}(K)}\,I_f(\mu)=:\widehat{w}_{f,L_K}(K),
 \end{equation}
and hence
\[-\infty<-2M_KL_K\leqslant\widehat{w}_f(K)\leqslant0.\]

 Choose a (minimizing) sequence $(\mu_j)\subset\mathcal E^+_{L_K}(K)$ such that
 \[\lim_{j\to\infty}\,I_f(\mu_j)=\widehat{w}_{f,L_K}(K).\]
 Being vaguely bounded in consequence of (\ref{LL}), $(\mu_j)$ is vaguely relatively compact \cite[Section~III.1,
 Proposition~15]{B2}, and so there is a subsequence $(\mu_{j_k})$ converging vaguely to some $\mu_0\in\mathfrak M^+(K)$. (Here
 we have used the facts that the vague topology on $\mathfrak M$ is first-countable, see Section~\ref{sec-intr1}, and that $\mathfrak M^+(K)$ is vaguely closed \cite[Section~III.2, Proposition~6]{B2}.) The energy $I(\cdot)$ being vaguely l.s.c.\ on $\mathfrak M^+$ \cite[Lemma~2.2.1(e)]{F1},
 \[\|\mu_0\|^2\leqslant\liminf_{k\to\infty}\,\|\mu_{j_k}\|^2,\]
 and so $\mu_0\in\mathcal E^+_f(K)$.
 Furthermore, by \cite[Section~IV.4, Corollary~3 to Proposition~5]{B2},
 \[\int U^\omega\,d\mu_0\geqslant\limsup_{k\to\infty}\,\int U^\omega\,d\mu_{j_k}\in(-\infty,\infty),\]
 $U^\omega$ being bounded and u.s.c.\ on the (compact) set $K$, see (b). This altogether gives
 \[\widehat{w}_{f}(K)\leqslant I_f(\mu_0)\leqslant\liminf_{k\to\infty}\,I_f(\mu_{j_k})=\widehat{w}_{f,L_K}(K),\]
 which combined with (\ref{L}) shows that $\mu_0$ serves as the solution $\widehat{\omega}^K$ to problem (\ref{W}). Moreover, the same $\widehat{\omega}^K$ solves the $f$-weighted minimum energy problem (\ref{L}), for
 \[\widehat{\omega}^K(X)=\mu_0(X)\leqslant\liminf_{k\to\infty}\,\mu_{j_k}(X)\leqslant L_K,\]
 the mapping $\nu\mapsto\nu(X)$ being vaguely l.s.c.\ on $\mathfrak M^+$ \cite[Section~IV.1, Proposition~4]{B2}.

We next aim to show that a constant $L_K$, satisfying (\ref{L}), can be chosen to be independent of
$K\in\mathfrak C_A$. To this end, we first conclude from (\ref{def2}) with $A:=K$ that
 \begin{equation*}
  U^{\widehat{\omega}^K}\leqslant U^\omega\quad\text{on $\mathfrak S:=S(\widehat{\omega}^K)$},
 \end{equation*}
$U^\omega$ being u.s.c.\ on $K$ by virtue of (b), whereas $U^{\widehat{\omega}^K}$ being l.s.c.\ on $X$. This combined with (\ref{def1}) gives $U^{\widehat{\omega}^K}=U^\omega$ n.e.\ on $\mathfrak S$, hence $\gamma_{\mathfrak S}$-a.e.\ (Lemma~\ref{l1}), where $\gamma_{\mathfrak S}\in\mathcal E^+(\mathfrak S)$ denotes the capacitary measure on $\mathfrak S$.\footnote{See Fuglede \cite[Theorem~2.5]{F1} for the concept of {\it capacitary measure} on a compact set and its properties. (For any compact $Q\subset X$, $c(Q)<\infty$ by the energy principle, and hence $\gamma_Q\in\mathcal E^+(Q)$ does indeed exist. If $h=1$, i.e.\ Frostman's maximum principle holds, then the capacitary measure is usually referred to as the {\it equilibrium measure}.)} Since $U^{\gamma_{\mathfrak S}}\geqslant1$ n.e.\ on $\mathfrak S$, hence $\widehat{\omega}^K$-a.e., applying Lebesgue--Fubini's theorem \cite[Section~V.8, Theorem~1]{B2} yields
 \begin{align*}
   \widehat{\omega}^K(X)\leqslant
   \int U^{\gamma_{\mathfrak S}}\,d\widehat{\omega}^K=\int U^{\widehat{\omega}^K}\,d\gamma_{\mathfrak S}=\int
   U^\omega\,d\gamma_{\mathfrak S}=
   \int U^{\gamma_{\mathfrak S}}\,d\omega\leqslant\int U^{\gamma_{\mathfrak S}}\,d\omega^+.
    \end{align*}
As $U^{\gamma_{\mathfrak S}}\leqslant1$ on $S(\gamma_{\mathfrak S})$, we therefore get, by $h$-Ugaheri's maximum principle (see (c)),
\begin{equation}\label{La}
\widehat{\omega}^K(X)\leqslant h\omega^+(X)=:L\in(0,\infty)\quad\text{for all $K\in\mathfrak C_A$},\end{equation}
$\omega^+(X)$ being finite according to (a). Thus,
\[\widehat{\omega}^K\in\mathcal E^+_{L}(K)\subset\mathcal E^+_f(K),\] $\mathcal E^+_{L}(K)$ being introduced by (\ref{LL}) with $L$ in place of $L_K$. This implies
\[\widehat{w}_{f,L}(K)\leqslant I_f(\widehat{\omega}^K)=\widehat{w}_f(K)\leqslant\widehat{w}_{f,L}(K),\]
and so $\widehat{\omega}^K$ also solves the problem of minimizing $I_f(\mu)$, where $\mu$ ranges over $\mathcal E^+_{L}(K)$.

We also remark that
\begin{equation}\label{KA}
\widehat{w}_f(K)=\widehat{w}_{f,L}(K)\in[-2M_AL,0]\quad\text{for all $K\in\mathfrak C_A$,}
\end{equation}
the constants $M_A,L\in(0,\infty)$ being introduced by (\ref{MA}) and (\ref{La}), respectively.

To prove the existence of the solution $\widehat{\omega}^A$ to problem (\ref{W}) for noncompact $A$, we first note that the net $\bigl(\widehat{w}_{f,L}(K)\bigr)_{K\in\mathfrak C_A}$ decreases, and moreover, by virtue of (\ref{KA}),
\begin{equation}\label{lim}
\infty<\lim_{K\uparrow A}\,\widehat{w}_{f,L}(K)\leqslant0.
\end{equation}

For any compact $K,K'\subset A$ such that $K\subset K'$,
\[(\widehat{\omega}^K+\widehat{\omega}^{K'})/2\in\mathcal E_L^+(K'),\]
whence
\[\|\widehat{\omega}^K+\widehat{\omega}^{K'}\|^2-4\int
U^\omega\,d(\widehat{\omega}^K+\widehat{\omega}^{K'})\geqslant4\widehat{w}_{f,L}(K')=4I_f(\widehat{\omega}^{K'}),\]
which yields, by applying the parallelogram identity to $\widehat{\omega}^K,\widehat{\omega}^{K'}\in\mathcal E^+$,
\begin{equation}\label{fund}
 \|\widehat{\omega}^K-\widehat{\omega}^{K'}\|^2\leqslant2I_f(\widehat{\omega}^K)-2I_f(\widehat{\omega}^{K'}).
\end{equation}
Noting from (\ref{lim}) that the net $\bigl(I_f(\widehat{\omega}^K)\bigr)_{K\in\mathfrak C_A}$, being equal to $\bigl(\widehat{w}_{f,L}(K)\bigr)_{K\in\mathfrak C_A}$, is Cauchy in $\mathbb R$, we infer from
(\ref{fund}) that the net $(\widehat{\omega}^K)_{K\in\mathfrak C_A}$ is strong Cauchy in $\mathcal E^+(A)$. As $\mathcal E^+(A)$ is strongly complete (Section~\ref{sec-intr1}), there exists $\zeta\in\mathcal E^+(A)$ such that
\begin{equation}\label{conv}
 \widehat{\omega}^K\to\zeta\quad\text{strongly (hence vaguely) in $\mathcal E^+(A)$ as $K\uparrow A$.}
\end{equation}
Moreover, $\zeta\in\mathcal E^+_L(A)$, since $\mu\mapsto\mu(X)$ is vaguely l.s.c.\ on $\mathfrak M^+$; and so the $\zeta$-mea\-s\-urable set $A$ is actually $\zeta$-integrable \cite[Section~IV.5, Corollary~1 to Theorem~5]{B2}.

We claim that the same $\zeta$ serves as the solution $\widehat{\omega}^A$ to problem (\ref{W}). As shown above, cf.\ (\ref{st1}), this will follow if we prove both (\ref{def1}) and (\ref{def2}) for $\zeta$ in place of $\widehat{\omega}^A$.

Clearly, (\ref{def1}) only needs to be verified for any given compact subset $K_0$ of $A$. The strong topology on $\mathcal E^+$ being first-countable, in view of (\ref{conv}) there is a subsequence
$(\widehat{\omega}^{K_j})_{j\in\mathbb N}$ of the net $(\widehat{\omega}^K)_{K\in\mathfrak C_A}$ such that $K_j\supset K_0$ for all $j$,\footnote{If this does not hold, we replace $\mathfrak C_A$ by its subset $\mathfrak C_A':=\{K\cup K_0:\ K\in\mathfrak C_A\}$ with the partial order relation inherited from $\mathfrak C_A$, and then apply to $\mathfrak C_A'$ the same arguments as just above.\label{FFot}} and
\begin{equation}\label{J}
\widehat{\omega}^{K_j}\to\zeta\quad\text{strongly (hence vaguely) in $\mathcal E^+(A)$ as $j\to\infty$.}
\end{equation}
Passing if necessary to a subsequence and changing notations, we derive from (\ref{J}), by exploiting Lemma~\ref{l2}, that
\begin{equation}\label{JJ}U^\zeta=\lim_{j\to\infty}\,U^{\widehat{\omega}^{K_j}}\quad\text{n.e.\ on $X$}.\end{equation}
Applying now (\ref{def1}) to each of $\widehat{\omega}^{K_j}$, and then letting $j\to\infty$, we infer from (\ref{JJ}) that (\ref{def1}) (for $\zeta$ in place of $\widehat{\omega}^A$) does indeed hold n.e.\ on $K_0$,\footnote{Here we have utilized the countable subadditivity of inner capacity on universally measurable sets (see Fuglede \cite[Lemma~2.3.5]{F1}, cf.\ also Lemma~\ref{str} above).} whence n.e.\ on $A$.

It remains to prove (\ref{def2}) for $\zeta$ in place of $\widehat{\omega}^A$. The set $A$ being $\zeta$-integrable, there is a countable union $A'$ of pairwise disjoint compact subsets of $A$ such that $A\setminus A'$ is $\zeta$-neg\-ligible, cf.\ \cite[Section~IV.4, Corollary~2 to Theorem~4]{B2}. Thus, by virtue of \cite[Section~IV.4, Proposition~9]{B2}, it is enough to verify the equality
\begin{equation}\label{kk}
U^\zeta=U^\omega\quad\text{$\zeta$-a.e.\ on $K$},
\end{equation}
where $K$ is any given compact subset of $A'$. As in footnote~\ref{FFot}, there is no loss of generality in assuming $K\subset K_j$ for all $j\in\mathbb N$, the sets $K_j$ being the same as above.

It is clear from (\ref{def2}) applied to each of those $K_j$ that
\begin{equation}\label{Kj}
  U^{\widehat{\omega}^{K_j}}\leqslant U^\omega\quad\text{on $S(\widehat{\omega}^{K_j})$},
 \end{equation}
$U^\omega$ being u.s.c.\ on $K_j$ by (b), while $U^{\widehat{\omega}^{K_j}}$ being l.s.c.\ on $X$.
Since $(\widehat{\omega}^{K_j})$ converges to $\zeta$ vaguely, see (\ref{J}), for every
$x\in S(\zeta)\cap K$ there exist a subsequence $(K_{j_k})$ of $(K_j)$ and points $x_{j_k}\in S(\widehat{\omega}^{K_{j_k}})$ such that
$x_{j_k}\to x$ as $k\to\infty$. Thus, by (\ref{Kj}),
\[U^{\widehat{\omega}^{K_{j_k}}}(x_{j_k})\leqslant U^\omega(x_{j_k})\quad\text{for all $k\in\mathbb N$}.\]
Letting here $k\to\infty$, by the upper semicontinuity of $U^\omega$ on the compact subsets of $A$ and the lower semicontinuity of the mapping
$(x,\mu)\mapsto U^\mu(x)$ on $X\times\mathfrak M^+$, $\mathfrak M^+$ being equipped with the vague
topology \cite[Lemma~2.2.1(b)]{F1}, we obtain
\[U^\zeta(x)\leqslant U^\omega(x)\quad\text{for all $x\in S(\zeta)\cap K$},\]
which combined with (\ref{def1}) gives (\ref{kk}), whence (\ref{def2}) (for $\zeta$ in place of $\widehat{\omega}^A$).

This implies that
\begin{equation}\label{xi'}
\zeta=\widehat{\omega}^A,
\end{equation}
thereby completing the proof of the whole theorem.

\begin{corollary}\label{H}
 If {\rm(a)--(c)} are fulfilled, then
 \[\widehat{\omega}^A(X)\leqslant h\omega^+(X)\in(0,\infty),\]
 $h$ being the constant appearing in the $h$-Ugaheri maximum principle.
 \end{corollary}

\begin{proof}
This is seen from (\ref{La}), (\ref{conv}), and (\ref{xi'}) by use of the vague lower semicontinuity of the mapping $\nu\mapsto\nu(X)$ on $\mathfrak M^+$.
\end{proof}

\subsection{Proof of Theorem~\ref{th1'}} It runs precisely in the same manner as the proof of Theorem~2.5 in \cite{Z-Expo}, dealing with $\omega\in\mathcal E$, the only difference being in applying Theorem~\ref{th1} of the present paper in place of \cite[Theorem~2.1]{Z-Expo}.

\section{Convergence of pseudo-balayage for monotone families of sets}

\begin{theorem}\label{conv=ps}$\widehat{\omega}^K\to\widehat{\omega}^A$ strongly and vaguely in $\mathcal E^+(A)$ as $K\uparrow A$. If $A$ is Borel, then the same remains valid with $\widehat{\omega}^K$ and $\widehat{\omega}^A$ replaced by $\widehat{\omega}^{\,*K}$ and $\widehat{\omega}^{\,*A}$, respectively.
\end{theorem}

\begin{proof}On account of Theorem~\ref{th1'}, the latter claim is reduced to the former.
If $\omega\in\mathfrak M$ meets (a)--(c), then the former claim follows, in turn, by substituting (\ref{xi'}) into (\ref{conv}). For the remaining case $\omega\in\mathcal E$, see \cite[Theorem~4.1]{Z-Expo}.\end{proof}

\begin{theorem}\label{conv=ps'}Consider a decreasing sequence $(A_j)$ of quasiclosed sets with the intersection $A$ of nonzero inner capacity, and a measure $\omega\in\mathfrak M$ such that either $\omega\in\mathcal E$ or {\rm(a)}--{\rm(c)} are fulfilled with $A_1$ in place of $A$. Then
\begin{equation}\label{e-conv}
\widehat{\omega}^{A_j}\to\widehat{\omega}^A\quad\text{strongly and vaguely in $\mathcal E^+$ as $K\uparrow A$}.
\end{equation}
If moreover the sets $A_j$, $j\in\mathbb N$, are Borel, then {\rm(\ref{e-conv})} remains valid with $\widehat{\omega}^{A_j}$ and $\widehat{\omega}^A$ replaced by $\widehat{\omega}^{\,*A_j}$ and $\widehat{\omega}^{\,*A}$, respectively.
\end{theorem}

\begin{proof} To verify the former part of the theorem, note that $\widehat{\omega}^{A_j}$, $j\in\mathbb N$, and $\widehat{\omega}^A$ do exist by virtue of Theorem~\ref{th1}, a countable intersection of quasiclosed sets being likewise quasiclosed \cite[Lemma~2.3]{F71}.\footnote{Recall that for quasiclosed $A\subset X$, $\mathcal E^+(A)$ is strongly closed \cite[Theorem~2.13]{Z-Oh}.} Moreover, (\ref{e-conv}) does indeed hold whenever $\omega\in\mathcal E$, see \cite[Theorem~4.2]{Z-Expo}.
It thus remains to consider the case of $\omega$ satisfying (a)--(c) with $A_1$ in place of $A$. Clearly, $\widehat{w}_f(A_j)$ increases as $j\to\infty$, and does not exceed $\widehat{w}_f(A)$, cf.\ (\ref{W}). Together with (\ref{e-i}), this gives
\begin{equation*}
-\infty<\lim_{j\to\infty}\,I_f(\widehat{\omega}^{A_j})\leqslant0,
\end{equation*}
and hence the sequence $\bigl(I_f(\widehat{\omega}^{A_j})\bigr)$ is Cauchy in $\mathbb R$. As $\widehat{\omega}^{A_{j+1}}\in\mathcal E^+_f(A_{j+1})\subset\mathcal E^+_f(A_j)$, while $\widehat{\omega}^{A_j}$ minimizes $I_f(\mu)$ over $\mu\in\mathcal E^+_f(A_j)$, we conclude, by utilizing the convexity of the class
$\mathcal E^+_f(A_j)$ and the parallelogram identity applied to $\widehat{\omega}^{A_j}$ and $\widehat{\omega}^{A_{j+1}}$, that
\[\|\widehat{\omega}^{A_{j+1}}-\widehat{\omega}^{A_j}\|^2\leqslant2I_f(\widehat{\omega}^{A_{j+1}})-2I_f(\widehat{\omega}^{A_j}),\]
and hence the sequence $(\widehat{\omega}^{A_j})_{j\geqslant k}$ is strong Cauchy in
$\mathcal E^+(A_k)$ for each $k\in\mathbb N$. The cone $\mathcal E^+(A_k)$ being strongly closed, hence strongly complete (Section~\ref{sec-intr1}), there exists a measure $\mu_0$ which belongs to $\mathcal E^+(A_k)$ for each $k\in\mathbb N$, and such that
\begin{equation}\label{LL''}
\widehat{\omega}^{A_j}\to\mu_0\quad\text{strongly (hence, vaguely) in $\mathcal E^+$}.
\end{equation}
We assert that
\begin{equation}\label{EE''}
 \mu_0=\widehat{\omega}^A,
\end{equation}
which substituted into (\ref{LL''}) would have implied (\ref{e-conv}).

To this end, we first note that, being the (countable) union of $\mu_0$-negligible $(A_k)^c$, the set $A^c$ is likewise $\mu_0$-neg\-li\-gi\-b\-le \cite[Section~IV.2, Proposition~4]{B2}, and therefore
\begin{equation}\label{mu0}
\mu_0\in\mathcal E^+(A).
\end{equation}
Moreover, by virtue of Corollary~\ref{H}, $\widehat{\omega}^{A_j}(X)\leqslant h\omega^+(X)<\infty$ for all $j$.
Since $\widehat{\omega}^{A_j}\to\mu_0$ vaguely, while $\nu\mapsto\nu(X)$ is vaguely l.s.c.\ on $\mathfrak M^+$, the measure $\mu_0$ must, therefore, be {\it bounded}. Combined with (\ref{mu0}), this shows that, actually,
\begin{equation}\label{mu0f}
\mu_0\in\mathcal E^+_f(A).
\end{equation}
Yet another conclusion is that the set $A$ is $\mu_0$-integrable \cite[Section~IV.5, Corollary~1 to Theorem~5]{B2}, the $\mu_0$-measurability of $A$ being clear from (\ref{mu0}).

In view of (\ref{def1}) applied to $\widehat{\omega}^{A_j}$, $U^{\widehat{\omega}^{A_j}}\geqslant U^\omega$ n.e.\ on $A_j$,
hence n.e.\ on the (smaller) set $A$, which yields, by exploiting (\ref{LL''}) and Lemmas~\ref{str} and \ref{l2},
\begin{equation}\label{mu0'}U^{\mu_0}\geqslant U^\omega\quad\text{n.e.\ on $A$}.\end{equation}
On account of Theorem~\ref{th1}(iii$_1$), this together with (\ref{mu0f}) implies that (\ref{EE''}) will be established once we verify the equality
\begin{equation}\label{mu00}U^{\mu_0}=U^\omega\quad\text{$\mu_0$-a.e.}\end{equation}

As $A$ is $\mu_0$-integrable, utilizing \cite[Section~IV.4, Corollary~2 to Theorem~4]{B2} shows that there exists $A'\subset A$, a countable union of pairwise disjoint compact sets, such that $A\setminus A'$ is $\mu_0$-negligible. This implies, by use of \cite[Section~IV.4, Proposition~9]{B2}, that (\ref{mu00}) only needs to be verified for any given compact $K\subset A'$.

Applying (\ref{def2}) to each of the sets $A_j$, we obtain
\begin{equation}\label{Aj}
  U^{\widehat{\omega}^{A_j}}\leqslant U^\omega\quad\text{on $S(\widehat{\omega}^{A_j})\cap K$},
 \end{equation}
$U^\omega$ being u.s.c.\ on the compact subsets of $A_1$, while $U^{\widehat{\omega}^{A_j}}$ being l.s.c.\ on $X$. As $(\widehat{\omega}^{A_j})$ converges to $\mu_0$ vaguely, for every
$x\in S(\mu_0)\cap K$ there exist a subsequence $(A_{j_k})$ of $(A_j)$ and points $x_{j_k}\in S(\widehat{\omega}^{A_{j_k}})\cap K$ such that
$x_{j_k}\to x$ when $k\to\infty$. Thus, by (\ref{Aj}),
\[U^{\widehat{\omega}^{A_{j_k}}}(x_{j_k})\leqslant U^\omega(x_{j_k})\quad\text{for all $k\in\mathbb N$}.\]
Letting here $k\to\infty$, in view of the upper semicontinuity of $U^\omega$ on $K$ and the lower semicontinuity of the mapping
$(x,\mu)\mapsto U^\mu(x)$ on $X\times\mathfrak M^+$, $\mathfrak M^+$ being equipped with the vague
topology \cite[Lemma~2.2.1(b)]{F1}, we get
\[U^{\mu_0}(x)\leqslant U^\omega(x)\quad\text{for all $x\in S(\mu_0)\cap K$},\]
which combined with (\ref{mu0'}) gives (\ref{mu00}), whence (\ref{e-conv}).

Substituting (\ref{io}) into (\ref{e-conv}) we obtain the latter claim of the theorem, thereby completing the whole proof. \end{proof}

\section{Pseudo-balayage in the inner Gauss variational problem}\label{sec-appl}

The aim of the rest of this study is to show that the concept of inner pseudo-bal\-ayage, introduced by means of Theorem~\ref{th1}, serves as a powerful tool in the inner Gauss
variational problem, the problem on the existence of $\lambda_{A,f}\in\breve{\mathcal E}^+_f(A)$ with
\begin{equation}\label{G}
 I_f(\lambda_{A,f})=\inf_{\mu\in\breve{\mathcal E}^+_f(A)}\,I_f(\mu)=:w_f(A).
\end{equation}
Here $I_f(\mu)$ is the energy of $\mu$ evaluated in the presence of the external field $f=-U^\omega$, see (\ref{If}), referred to as the Gauss functional or the $f$-weighted energy, while\footnote{Since $f$ is $\mu$-integrable for every $\mu\in\breve{\mathcal E}^+(A)$, we actually have
\[\breve{\mathcal E}^+_f(A)=\breve{\mathcal E}^+(A):=\bigl\{\mu\in\mathcal E^+(A):\ \mu(X)=1\bigr\}.\] Note that $\breve{\mathcal E}^+(A)\ne\varnothing$, for $\mathcal E^+(A)\ne\{0\}$, cf.\ (\ref{iff}) and (\ref{non0}).}
\[\breve{\mathcal E}^+_f(A):=\bigl\{\mu\in\mathcal E^+_f(A):\ \mu(X)=1\bigr\}.\]
For the bibliography on this problem, see \cite{BHS,Dr0,O,ST,Z-Rarx,Z-Oh} and references therein.

Since $f$ is $\mu$-integrable for each $\mu$ from the (nonempty) class $\breve{\mathcal E}^+_f(A)$, we have
\begin{equation}\label{ww}
-\infty<\widehat{w}_f(A)\leqslant w_f(A)<\infty,
\end{equation}
the first inequality being clear from (\ref{e-i}). Thus $w_f(A)$ is {\it finite}, which enables us to prove, by use of the convexity of the class $\breve{\mathcal E}^+_f(A)$ and the pre-Hilbert structure on the space $\mathcal E$, that the solution $\lambda_{A,f}$ to problem (\ref{G}) is {\it unique} (cf.\ \cite[Lemma~6]{Z5a}).

As for the existence of $\lambda_{A,f}$, assume for a moment that $A:=K$ is compact. If moreover $f$ is bounded and l.s.c.\ on $K$ (as in the case in question), then $\lambda_{K,f}$ does indeed exist, because then the class
$\breve{\mathfrak M}^+(K):=\{\mu\in\mathfrak M^+(K):\ \mu(X)=1\}$ is vaguely compact, cf.\ \cite[Section~III.1.9, Corollary~3]{B2}, while $I_f(\mu)$ is vaguely l.s.c.\ on $\mathfrak M^+(K)$, cf.\ \cite[Section~IV.1, Proposition~4]{B2}. (See \cite[Theorem~2.6]{O}.) But if any of the above two assumptions is not fulfilled, then these arguments, based on the vague topology only, fail down, and the problem becomes "rather difficult" (Ohtsuka \cite[p.~219]{O}).

Our analysis of the inner Gauss variational problem for $A$ and $f=-U^\omega$, indicated in Section~\ref{sec-intr1} above, is based on the concept of the inner pseudo-balayage $\widehat{\omega}^A$, introduced in the present study, and is mainly performed with the aid of the approach originated in our recent work \cite{Z-Oh}. However, \cite{Z-Oh} was only concerned with external fields created by {\it positive} measures of {\it finite} energy, and exactly this circumstance made it possible to exploit efficiently the two topologies~--- strong and vague, whereas the treatment of the problem for {\it signed} $\omega$ whose energy might be {\it infinite}, needs the involvement of more delicate arguments. Because of this obstacle, we have to impose on the objects in question some additional requirements.

Referring to \cite[Section~6.1]{Z-Expo} for necessary and/or sufficient conditions for the solvability of problem (\ref{G}) in the case $\omega\in\mathcal E$, we assume in the sequel that, along with the requirements (a)--(c), the following (d)--(f) hold true:
\begin{itemize}
  \item[(d)] {\it $\kappa(x,y)$ is continuous for $x\ne y$.}
  \item[(e)] {\it When $y\to\infty_X$, $\kappa(\cdot,y)\to0$ uniformly on compact subsets of $X$.}
  \item[(f)] {\it $\omega$ is compactly supported in $\overline{A}^c$, where $\overline{A}:={\rm Cl_XA}$.}
\end{itemize}
(Here $\infty_X$ denotes the Alexandroff point of $X$, see \cite[Section~I.9.8]{B1}.) These permanent assumptions will not be repeated henceforth.

Then necessary and sufficient conditions for the solvability of problem (\ref{G}) can be given in the following rather simple form.

\begin{theorem}\label{th-solv1} For the solution $\lambda_{A,f}$ to exist, it is necessary and sufficient that
\begin{equation*}
c_*(A)<\infty\quad\text{or}\quad\widehat{\omega}^A(X)\geqslant1.
\end{equation*}
If moreover equality prevails in the latter inequality, then actually
\[\lambda_{A,f}=\widehat{\omega}^A.\]
\end{theorem}

\begin{corollary}\label{cor-qu}
 $\lambda_{A,f}$ does exist whenever $A$ is quasicompact.
\end{corollary}

\begin{proof}
This is obvious from Theorem~\ref{th-solv1} because a quasicompact set, being approximated in outer capacity by compact sets, is of finite outer (hence inner) capacity.
\end{proof}

\begin{corollary}\label{th3-cor2}If $c_*(A)=\infty$ and $\omega=-\omega^-$, then problem {\rm(\ref{G})} is unsolvable.\end{corollary}

\begin{proof} Since then $\widehat{\omega}^A=0$ (Remark~\ref{rem2}), this follows directly from Theorem~\ref{th-solv1}.
 \end{proof}

\begin{corollary}\label{th-unsolv}If $c_*(A)=\infty$ and $\omega^+(X)<1/h$, $h$ being the constant appearing in the $h$-Ugaheri maximum principle, then problem~{\rm(\ref{G})} is unsolvable.\footnote{Observe that Corollary~\ref{th-unsolv} remains in force when  $\omega^-(X)\in[0,\infty]$ becomes arbitrarily large, which however agrees with our physical intuition.}
\end{corollary}

\begin{proof} This follows by combining Theorem~\ref{th-solv1} with Corollary~\ref{H}.
 \end{proof}

\begin{remark}\label{rem-rem}
It is clear from the assumptions (d)--(f) that both $U^{\omega^+}$ and $U^{\omega^-}$ are bounded and continuous on $\overline{A}$. Furthermore, by virtue of (e) and (f),
\begin{equation}\label{limU}
\lim_{x\to\infty_X}\,U^{\omega^\pm}(x)=0.
\end{equation}
This implies, in particular, that under the permanent requirements (d)--(f), both (a) and (b) are necessarily fulfilled, and hence can be omitted.
\end{remark}

\begin{example}\label{rem:clas1}
All the permanent requirements (a)--(f) do hold, for instance, for the following kernels:
\begin{itemize}
  \item[\checkmark] The $\alpha$-Riesz kernels $\kappa_\alpha$ of arbitrary order $\alpha\in(0,n)$ on $\mathbb R^n$, $n\geqslant2$.
  \item[\checkmark] The $\alpha$-Green kernels, $0<\alpha\leqslant2$, on an open bounded subset of $\mathbb R^n$, $n\geqslant2$.
\end{itemize}
\end{example}

\begin{example}
Consider the $\alpha$-Riesz kernel $\kappa_\alpha$ of order $0<\alpha\leqslant2$, $\alpha<n$ on $\mathbb R^n$, $n\geqslant2$, a closed set $A\subset\mathbb R^n$ of nonzero capacity, and a positive measure $\omega$, $\omega\ne0$, compactly supported in $A^c$.

\begin{corollary}\label{corR}
Then, the following assertions on the solvability hold true.\footnote{Being actually valid under much more general assumptions, see \cite{Z-Rarx} (Theorems~2.1, 2.13(b$_1$), and 2.15), Corollary~\ref{corR} is presented here in order to illustrate Theorem~\ref{th-solv1}. For another illustration of these results, now given in geometrical terms, see also \cite[Example~2.24]{Z-Rarx}.}
\begin{itemize}
\item[{\rm(i$_2$)}] If $c(A)<\infty$, then $\lambda_{A,f}$ does exist for any $\omega$.
\item[{\rm(ii$_2$)}] If $A$ is not $\alpha$-thin at infinity, then\footnote{By Kurokawa and Mizuta \cite[Definition~3.1]{KM}, $A$ is said to be {\it $\alpha$-thin at infinity} if
\begin{equation*}
 \sum_{j\in\mathbb N}\,\frac{c(A_j)}{q^{j(n-\alpha)}}<\infty,
 \end{equation*}
where $q\in(1,\infty)$ and $A_j:=A\cap\{x\in\mathbb R^n:\ q^j\leqslant|x|<q^{j+1}\}$. See also Zorii \cite[Definition~2.1]{Z-bal2} and Doob \cite[pp.~175--176]{Doob}, the latter pertaining to $\alpha=2$.}
\[\lambda_{A,f}\text{\ exists}\iff\omega(\mathbb R^n)\geqslant1.\]
\item[{\rm(iii$_2$)}] Assume that $A$ is $\alpha$-thin at infinity, $c(A)=\infty$, and $\omega(\mathbb R^n)\leqslant1$. If moreover $A^c$ is connected unless $\alpha<2$, then $\lambda_{A,f}$ fails to exist.
\end{itemize}
\end{corollary}

\begin{proof} We begin by observing that for $\kappa$ and $\omega$ in question, $\widehat{\omega}^A$ is actually the $\alpha$-Riesz balayage $\omega^A$ (cf.\ Examples~\ref{rem:clas}(i), \ref{rem:clas1} and Remark~\ref{rem1}), and that (i$_2$) is an immediate corollary to Theorem~\ref{th-solv1}.

Assuming now that $A$ is not $\alpha$-thin at infinity, we get $c(A)=\infty$, and hence, again by Theorem~\ref{th-solv1}, $\lambda_{A,f}$ exists if and only if $\omega^A(\mathbb R^n)\geqslant1$. As $A$ is not $\alpha$-thin at infinity,
$\mu^A(\mathbb R^n)=\mu(\mathbb R^n)$ for every $\mu\in\mathfrak M^+$ (see \cite[Corollary~5.3]{Z-bal2}), whence (ii$_2$).

Under the assumptions of (iii$_2$), $\omega^A(\mathbb R^n)<\omega(\mathbb R^n)\leqslant1$ by \cite[Theorem~8.7]{Z-bal} (cf.\ also \cite[Definition~2.1]{Z-bal2}), which according to Theorem~\ref{th-solv1} completes the proof.
\end{proof}
\end{example}

\section{Proof of Theorem~\ref{th-solv1}}

\subsection{Preparatory results} The following criterion for the existence of the solution $\lambda_{A,f}$ to problem (\ref{G}) was discovered in the author's old work \cite{Z5a}. Nevertheless,  it is still of importance, serving as an important tool in various researches on this topic (see, in particular, the present study as well as \cite{Dr0}, \cite{Z-Rarx}--\cite{Z-Expo}).

\begin{theorem}\label{th-ch2}For $\mu\in\breve{\mathcal E}^+(A)$ to be the {\rm(}unique{\rm)} solution
$\lambda_{A,f}$ to problem {\rm(\ref{G})}, it is necessary and sufficient that
\begin{equation}\label{ch-1}
U_f^\mu\geqslant\int U^\mu_f\,d\mu\quad\text{n.e.\ on $A$},
\end{equation}
or equivalently
\begin{equation}\label{ch-2}
U_f^\mu=\int U^\mu_f\,d\mu\quad\text{$\mu$-a.e.\ on $X$},
\end{equation}
where $U_f^\mu:=U^\mu+f$ is said to be the $f$-weighted potential of $\mu$.\footnote{For $f$ in question and for any $\nu\in\mathcal E^+$, $U_f^\nu$ is finite n.e.\ on $A$, which is obvious in view of the fact that $U^\nu$ is finite q.e.\ (hence, n.e.) on $X$, cf.\ \cite[Corollary to Lemma~3.2.3]{F1}.\label{f1}}
The {\rm(}finite{\rm)} constant
\begin{equation}\label{ch-3}c_{A,f}:=\int U_f^{\lambda_{A,f}}\,d\lambda_{A,f}\end{equation} is referred to as the inner $f$-weighted equilibrium constant for the set $A$.
\end{theorem}

\begin{proof}
 See \cite{Z5a} (Theorems~1, 2 and Proposition~1).
\end{proof}

\begin{definition}\label{def-min}
A net $(\mu_s)\subset\breve{\mathcal E}^+_f(A)$ is said to be {\it minimizing} if
\begin{equation}\label{min}
\lim_{s}\,I_f(\mu_s)=w_f(A);
\end{equation}
let $\mathbb M_f(A)$ stand for the (nonempty) set of all those $(\mu_s)$.\footnote{Since the proof of Theorem~\ref{th-solv1}  follows in general the scheme suggested in \cite{Z-Oh}, some of its parts are quite similar to those in \cite{Z-Oh,Z-Expo}, though  being performed under other assumptions on the external field under consideration. To make the proof of Theorem~\ref{th-solv1} clear and self-contained, we briefly represent some of those parts, correspondingly modified.}
\end{definition}

\begin{lemma}\label{l-extr}
There is the unique $\xi_{A,f}\in\mathcal E^+(A)$ such that for every $(\mu_s)\in\mathbb M_f(A)$,
\begin{equation}\label{e-extr}
\mu_s\to\xi_{A,f}\quad\text{strongly and vaguely in $\mathcal E^+(A)$};
\end{equation}
this $\xi_{A,f}$ is said to be the extremal measure in problem {\rm(\ref{G})}.
\end{lemma}

\begin{proof}
In a manner similar to that in \cite[Proof of Lemma~4.1]{Z-Oh}, one can see that for any $(\mu_s)_{s\in S}$ and $(\nu_t)_{t\in T}$
from $\mathbb M_f(A)$,
\begin{equation}\label{ST}
\lim_{(s,t)\in S\times T}\,\|\mu_s-\nu_t\|=0,
\end{equation}
$S\times T$ being the product of the directed sets $S$ and $T$ \cite[p.~68]{K}. Taking the two nets in (\ref{ST}) to be equal, we deduce that every $(\nu_t)_{t\in T}\in\mathbb M_f(A)$ is strong Cauchy. The cone $\mathcal E^+(A)$ being strongly closed, hence strongly complete (Section~\ref{sec-intr1}), $(\nu_t)_{t\in T}$ must converge strongly to some (unique) $\xi_{A,f}\in\mathcal E^+(A)$. The same $\xi_{A,f}$ also serves as the strong limit of any other $(\mu_s)_{s\in S}\in\mathbb M_f(A)$, which is obvious from (\ref{ST}). The strong topology on $\mathcal E^+$ being finer than the vague topology on $\mathcal E^+$ according to the perfectness of the kernel $\kappa$, $(\mu_s)_{s\in S}$ must converge to $\xi_{A,f}$ also vaguely.
\end{proof}

\begin{lemma}\label{l-extr1} For the extremal measure $\xi_{A,f}$, we have
\begin{align}\label{ext1}
  I_f(\xi_{A,f})&=w_f(A),\\
  \xi_{A,f}(X)&\leqslant1,\label{ext2}
\end{align}
whence
\begin{equation*}
\xi_{A,f}\in\mathcal E^+_f(A).
\end{equation*}
\end{lemma}

\begin{proof}
Fix $(\mu_s)\in\mathbb M_f(A)$; then $\mu_s\to\xi:=\xi_{A,f}$ strongly and vaguely, which gives
\begin{equation}\label{S}\lim_{s}\,\|\mu_s\|=\|\xi\|\end{equation}
as well as (\ref{ext2}), the mapping $\nu\mapsto\nu(X)$ being vaguely l.s.c.\ on $\mathfrak M^+$.
In view of (\ref{If}), (\ref{min}), and (\ref{S}), the remaining claim (\ref{ext1}) is reduced to showing that
\begin{equation}\label{SS}
\lim_{s}\,\int U^\omega\,d\mu_s=\int U^\omega\,d\xi.
\end{equation}

In fact, if $\omega\in\mathcal E$, then (\ref{SS}) is obvious because of
\[0\leqslant\lim_{s}\,\bigl|\langle\omega,\mu_s-\xi\rangle\bigr|\leqslant\|\omega\|\cdot\lim_{s}\,\|\mu_s-\xi\|=0,\]
which is derived from (\ref{e-extr}) and the Cauchy--Schwarz inequality, applied to $\omega$ and $\mu_s-\xi$, elements of the space $\mathcal E$; while otherwise, (\ref{SS}) follows by use of (d)--(f).

Indeed, for any given $\varepsilon>0$, there is a compact set $K\subset X$ such that $U^{\omega^\pm}(x)<\varepsilon$ for all $x\not\in K$ (cf.\ (\ref{limU})), and therefore, on account of (\ref{ext2}),\footnote{In (\ref{SSS}) as well as in (\ref{SSS'}), we have utilized \cite[Section~IV.4]{B2}, see Proposition~2 and Corollary~2 to Theorem~1 therein.\label{foot}}
\begin{equation}\label{SSS}
\biggl|\int U^\omega(x)1_{K^c}(x)\,d(\mu_s-\xi)(x)\biggr|<4\varepsilon\quad\text{for all $s\geqslant s_0$}.
\end{equation}
There is certainly no loss of generality in assuming that $K_0:=K\cap\overline{A}$ is nonempty. Since $A^c$ is $\mu_s+\xi$-negligible, (\ref{SSS}) remains valid with $K^c$ replaced by $K_0^c$.

As $U^{\omega^+}$, resp.\ $U^{\omega^-}$, is continuous on $\overline{A}$, the Tietze-Urysohn extension
theorem \cite[Theorem~0.2.13]{E2} implies that there exist positive $\varphi^+,\varphi^-\in C_0(X)$ such that
\begin{align*}\varphi^\pm(x)&=U^{\omega^\pm}(x)\quad\text{if $x\in K_0$},\\
\varphi^\pm(x)&\leqslant\varepsilon\quad\text{otherwise},
\end{align*}
which indicates, in turn, that for all $s$ large enough,
\begin{equation}\label{SSS'}\biggl|\int U^\omega(x)1_{K_0}(x)\,d(\mu_s-\xi)(x)\biggr|=\biggl|\int \bigl(\varphi-\varphi|_{K^c_0}\bigr)\,d(\mu_s-\xi)\biggr|<5\varepsilon,\end{equation}
where $\varphi:=\varphi^+-\varphi^-\in C_0(X)$. (Here we have used the fact that $\mu_s\to\xi$ vaguely; see also footnote~\ref{foot}.) This combined with (\ref{SSS}) gives (\ref{SS}), whence the lemma.
\end{proof}

\begin{corollary}\label{l-e-m}Problem {\rm(\ref{G})} is solvable if and only if equality prevails in {\rm(\ref{ext2})}, i.e.
\begin{equation}\label{e1}
 \xi_{A,f}(X)=1,
\end{equation}
and in the affirmative case,
\begin{equation}\label{es}
 \xi_{A,f}=\lambda_{A,f}.
\end{equation}
\end{corollary}

\begin{proof}Indeed, if (\ref{e1}) holds, then $\xi_{A,f}$ belongs to $\breve{\mathcal E}^+_f(A)$, which together with (\ref{ext1}) shows that $\xi_{A,f}$ serves as $\lambda_{A,f}$. Further, if $\lambda_{A,f}$ exists, then the trivial net $(\lambda_{A,f})$ obviously belongs to $\mathbb M_f(A)$, and hence converges strongly to $\xi_{A,f}$ (Lemma~\ref{l-extr}) as well as to $\lambda_{A,f}$. This implies (\ref{es}), the strong topology on $\mathcal E^+(A)$ being Hausdorff.\end{proof}

\begin{lemma}\label{l-xi}
For the extremal measure $\xi:=\xi_{A,f}\in\mathcal E^+_f(A)$, we have
\begin{align}\label{e-xi1}
U_f^\xi&\geqslant C_\xi\quad\text{n.e.\ on $A$},\\
U_f^\xi&=C_\xi\quad\text{$\xi$-a.e.\ on $X$},\label{e-xi2}
\end{align}
where
\begin{equation}\label{cxi}
 C_\xi:=\int U_f^\xi\,d\xi\in(-\infty,\infty).
\end{equation}
\end{lemma}

\begin{proof}
For any $K\in\mathfrak C_A$, $\lambda_{K,f}$ does exist, which is clear from Corollary~\ref{l-e-m} by use of the vague compactness of the class $\breve{\mathfrak M}^+(K)$, see \cite[Section~III.1.9, Corollary~3]{B2}.
We assert that those $\lambda_{K,f}$ form a minimizing net, i.e.
\begin{equation}\label{Kmin}
(\lambda_{K,f})_{K\in\mathfrak C_A}\in\mathbb M_f(A).
\end{equation}

On account of Definition~\ref{def-min}, (\ref{Kmin}) is in fact reduced to
\begin{equation}\label{contwf1}
 \lim_{K\uparrow A}\,w_f(K)=w_f(A).
\end{equation}
Since the net $\bigl(w_f(K)\bigr)_{K\in\mathfrak C_A}$ obviously decreases and has $w_f(A)$ as a lower bound, (\ref{contwf1}) will follow once we show that for any given $\mu\in\breve{\mathcal E}^+_f(A)$,
\begin{equation}\label{contwf1'}
I_f(\mu)\geqslant\lim_{K\uparrow A}\,w_f(K).
\end{equation}
Noting that $\mu(K)\uparrow1$ as $K\uparrow A$, and applying \cite[Lemma~1.2.2]{F1} to each of the
(positive, l.s.c., $\mu$-integrable) functions $\kappa$, $U^{\omega^+}$, and $U^{\omega^-}$, we get
\begin{equation*}
I_f(\mu)=\lim_{K\uparrow A}\,I_f(\mu|_K)=\lim_{K\uparrow A}\,I_f(\nu_K)\geqslant\lim_{K\uparrow A}\,w_f(K),
\end{equation*}
where $\nu_K:=\mu|_K/\mu(K)\in\breve{\mathcal E}^+(K)$ ($K\in\mathfrak C_A$ being large enough). This proves (\ref{contwf1'}), whence (\ref{Kmin}). Therefore, by virtue of Lemma~\ref{l-extr},
\begin{equation}\label{min-ext-c}
\lambda_{K,f}\to\xi_{A,f}\quad\text{strongly and vaguely in $\mathcal E^+(A)$ as $K\uparrow A$}.
\end{equation}

On the other hand, applying (\ref{ch-1})--(\ref{ch-3}) to each $K\in\mathfrak C_A$ gives
 \begin{align}\label{ch-1-k}
U_f^{\lambda_{K,f}}&\geqslant c_{K,f}\quad\text{n.e.\ on $K$},\\
\label{ch-2-k}
U_f^{\lambda_{K,f}}&=c_{K,f}\quad\text{$\lambda_{K,f}$-a.e.,}
\end{align}
where
\begin{equation}\label{ch-3-k}c_{K,f}:=\int U_f^{\lambda_{K,f}}\,d\lambda_{K,f}.\end{equation}
Letting $K\uparrow A$ in (\ref{ch-3-k}), we deduce from (\ref{S}) and (\ref{SS}) that
\begin{equation}\label{lll0}
\lim_{K\uparrow A}\,c_{K,f}=C_\xi,
\end{equation}
$C_\xi$ being introduced by means of (\ref{cxi}).

Fix $K_0\in\mathfrak C_A$. The strong topology on $\mathcal E^+$ being first-countable, one can choose a subsequence
$(\lambda_{K_j,f})_{j\in\mathbb N}$ of the net $(\lambda_{K,f})_{K\in\mathfrak C_A}$ such that
\begin{equation}\label{J0}
 \lambda_{K_j,f}\to\xi\quad\text{strongly (hence vaguely) in $\mathcal E^+(A)$ as $j\to\infty$,}
\end{equation}
cf.\ (\ref{min-ext-c}). There is certainly no loss of generality in assuming that
\[K_0\subset K_j\quad\text{for all $j$,}\]
for if not, we replace $K_j$ by $K_j':=K_j\cup K_0$; then, by the monotonicity of $\bigl(w_f(K)\bigr)$, the sequence
$(\lambda_{K_j',f})_{j\in\mathbb N}$ remains minimizing, and hence also converges strongly to $\xi$.

Due to the arbitrary choice of $K_0\in\mathfrak C_A$, (\ref{e-xi1}) will follow once we show that
\begin{equation}\label{JJ'}U^\xi_f\geqslant C_\xi\quad\text{n.e.\ on $K_0$}.\end{equation}
Passing if necessary to a subsequence and changing notations, we infer from (\ref{J0}), by making use of Lemma~\ref{l2}, that
\begin{equation}\label{JJ0}U^\xi=\lim_{j\to\infty}\,U^{\lambda_{K_j,f}}\quad\text{n.e.\ on $X$}.\end{equation}
Now, applying (\ref{ch-1-k}) to each of $K_j$, $j\in\mathbb N$, and then letting $j\to\infty$, in view of (\ref{lll0}) and (\ref{JJ0}) we obtain (\ref{JJ'}). (Here we have utilized the countable subadditivity of inner capacity on universally measurable sets \cite[Lemma~2.3.5]{F1}.)

Since $f$ is continuous on $\overline{A}$ (Remark~\ref{rem-rem}), $U_f^{\lambda_{K,f}}$, where $K\in\mathfrak C_A$, is l.s.c.\ on $\overline{A}$, which in view of (\ref{ch-2-k}) yields
\begin{equation}\label{2K}
U_f^{\lambda_{K,f}}\leqslant c_{K,f}\quad\text{on $S(\lambda_{K,f})$}.
\end{equation}
Since $(\lambda_{K_j,f})$ converges to $\xi$ vaguely, cf.\ (\ref{J0}), for every $x\in S(\xi)$ there exist a subsequence
$(K_{j_k})$ of the sequence $(K_j)$ and points $x_{j_k}\in S(\lambda_{K_{j_k},f})$, $k\in\mathbb N$, such that $x_{j_k}$ approach $x$ as $k\to\infty$. Thus, according to (\ref{2K}),
\[U_f^{\lambda_{K_{j_k},f}}(x_{j_k})\leqslant c_{K_{j_k},f}\quad\text{for all $k\in\mathbb
N$}.\]
Letting here $k\to\infty$ and utilizing (\ref{lll0}), the continuity of $f$ on $\overline{A}$, and the lower semicontinuity of the mapping $(x,\mu)\mapsto U^\mu(x)$ on
$X\times\mathfrak M^+$, $\mathfrak M^+$ being equipped with the vague topology \cite[Lemma~2.2.1(b)]{F1}, we arrive at the inequality
\[U_f^\xi\leqslant C_\xi\quad\text{on $S(\xi)$},\]
which together with (\ref{e-xi1}) gives
\[U_f^\xi=C_\xi\quad\text{n.e.\ on $S(\xi)\cap A$}.\]
Applying Lemma~\ref{l1} therefore results in (\ref{e-xi2}), $\xi$ being a positive measure of finite energy, concentrated on both $A$ (Lemma~\ref{l-extr}) and $S(\xi)$, and hence on $S(\xi)\cap A$.
\end{proof}

\subsection{Proof of Theorem~\ref{th-solv1}}

Let first $c_*(A)<\infty$. According to Corollary~\ref{l-e-m}, problem (\ref{G}) is solvable whenever the extremal measure $\xi:=\xi_{A,f}$, uniquely determined by means of Lemma~\ref{l-extr}, belongs to the class
$\breve{\mathcal E}^+(A)$. This does indeed hold in view of Theorem~\ref{th-str}, $\xi$ being the strong limit of a minimizing net $(\mu_s)\subset\breve{\mathcal E}^+_f(A)$.

\begin{theorem}\label{th-str}
If $c_*(A)<\infty$, $\breve{\mathcal E}^+(A)$ is strongly closed, hence strongly complete.
\end{theorem}

\begin{proof} This follows in the same manner as \cite[Theorem~7.1]{Z-Expo}, pertaining to quasiclosed $A$.
\end{proof}

Assume now that $\widehat{\omega}^A(X)=1$. As $\widehat{\omega}^A\in\mathcal E^+_f(A)$ according to Theorem~\ref{th1}, then actually $\widehat{\omega}^A\in\breve{\mathcal E}^+_f(A)$, whence $I_f(\widehat{\omega}^A)\geqslant w_f(A)$. Also note from (\ref{e-i}) and (\ref{ww}) that
\[I_f(\widehat{\omega}^A)=\widehat{w}_f(A)\leqslant w_f(A).\]
Putting this all together implies that the solution $\lambda_{A,f}$ to problem (\ref{G}) does indeed exist, and moreover $\lambda_{A,f}=\widehat{\omega}^A$.

We next aim to show that problem (\ref{G}) is unsolvable provided that
\begin{equation}\label{sm}
c_*(A)=\infty\quad\text{and}\quad\widehat{\omega}^A(X)<1.
\end{equation}
The space $X$ being $\sigma$-compact, whereas $c_*(A)=\infty$, there exist a sequence $(U_j)$ of relatively compact open subsets of $X$ which cover $X$, such that $\overline{U_j}\subset U_{j+1}$ for each $j$  \cite[Section~I.9, Proposition~15]{B1}, and a sequence $(K_j)$ of compact sets
$K_j\subset A\cap U_{j+1}$ such that $K_j\cap\overline{U_j}=\varnothing$ and $c(K_j)\geqslant j$ for each $j$. If
$\lambda_j:=\gamma_{K_j}/c(K_j)\in\breve{\mathcal E}^+(K_j)$ denotes the normalized capacitary measure on $K_j$, then
\begin{align}\label{31}
\|\lambda_j\|&\to0\quad\text{as $j\to\infty$},\\
\label{32}\lambda_j&\to0\quad\text{vaguely in $\mathcal E^+$ as $j\to\infty$},
\end{align}
where the latter is obvious from the fact that for any compact $Q\subset X$, $S(\lambda_j)\cap Q=\varnothing$ for all $j$
large enough. Noting that $\widehat{\omega}^A(X)<1$, cf.\ (\ref{sm}), define
\begin{equation}\label{nnu}
\mu_j:=\widehat{\omega}^A+q\lambda_j,\quad\text{where\ $q:=1-\widehat{\omega}^A(X)\in(0,1]$}.
\end{equation}
As $\widehat{\omega}^A,\lambda_j\in\mathcal E^+(A)$, we have $\mu_j\in\breve{\mathcal E}^+(A)$; hence, $\mu_j\in\breve{\mathcal E}^+_f(A)$ for all $j$, and so
\begin{equation}\label{nnu'}w_f(A)\leqslant\liminf_{j\to\infty}\,I_f(\mu_j).\end{equation}
On the other hand, (\ref{nnu}) implies by means of a straightforward verification that
\[I_f(\mu_j)\leqslant\widehat{w}_f(A)+q^2\|\lambda_j\|^2+2q\bigl\langle\widehat{\omega}^A,\lambda_j\bigr\rangle
+2q\int U^{\omega^-}\,d\lambda_j,\]
and applying (\ref{ww}), (\ref{limU}), and (\ref{31}) therefore gives, by use of the Cauchy--Schwarz inequality,
\begin{equation*}\limsup_{j\to\infty}\,I_f(\mu_j)\leqslant w_f(A).\end{equation*}
Combined with (\ref{nnu'}), this shows that the sequence $(\mu_j)$ is, actually, minimizing, and hence converges strongly and vaguely to the extremal measure $\xi$ (Lemma~\ref{l-extr}). On account of (\ref{32}) and (\ref{nnu}), this yields
$\xi=\widehat{\omega}^A$, whence $\xi(X)<1$ (cf.\ (\ref{sm})), which according to Corollary~\ref{l-e-m} proves the unsolvability of problem (\ref{G}).

It remains to verify the solvability of problem (\ref{G}) in the case where
\begin{equation}\label{Larger}
\widehat{\omega}^A(X)>1.
\end{equation}
We first show that
\begin{equation}\label{ne0}
C_\xi\ne0,
\end{equation}
where $C_\xi$ is given by (\ref{cxi}). Indeed, assuming to the contrary that $C_\xi=0$, we infer from Lemma~\ref{l-xi} that
\begin{align*}
U^\xi&\geqslant U^\omega\quad\text{n.e.\ on $A$},\\
U^\xi&=U^\omega\quad\text{$\xi$-a.e.\ on $X$}.
\end{align*}
According to Theorem~\ref{th1}(iii$_1$), then necessarily
\[\xi=\widehat{\omega}^A,\]
which is however impossible, since $\xi(X)\leqslant1$ while $\widehat{\omega}^A(X)>1$ by virtue of (\ref{ext2}) and (\ref{Larger}), respectively. The contradiction thus obtained proves (\ref{ne0}).

Now, integrating (\ref{e-xi2}) with respect to $\xi$, we get
\[\int U_f^\xi\,d\xi=C_\xi\cdot\xi(X),\]
which in view of (\ref{cxi}) and (\ref{ne0}) gives $\xi(X)=1$. According to Corollary~\ref{l-e-m}, this implies the solvability of problem (\ref{G}), thereby completing the proof of the theorem.

\section{Acknowledgements} This research was supported in part by a grant from the Simons Foundation, the USA (1030291, N.V.Z.).

\section{A data availability statement} This manuscript has no associated data.

\section{Funding and Competing interests} The author has no relevant financial or non-financial interests to disclose.

\end{document}